\documentclass[reqno,11pt]{amsart}

\usepackage{enumerate}

\usepackage[colorlinks=true]{hyperref}
\hypersetup{urlcolor=blue, linkcolor=blue, citecolor=red}
\usepackage[english]{babel}

\numberwithin{equation}{section}

\newtheorem{thm}{Theorem}[section]
\newtheorem{lem}[thm]{Lemma}
\newtheorem{prop}[thm]{Proposition}

\theoremstyle{definition}
\newtheorem{defi}[thm]{Definition}
\newtheorem{rem}[thm]{Remark}

\newcommand\R{{\mathbb R}}
\newcommand\C{{\mathbb C}}

\newcommand\N{{\mathbb N}}

\newcommand\Supp{{\mathrm{supp}}\, }

\newcommand\MScN[1]{\href{http://www.ams.org/mathscinet-getitem?mr=#1}{\nolinkurl{(#1)}}}
\newcommand\DOI[1]{\href{http://dx.doi.org/#1}{(doi: \nolinkurl{#1})}}
\newcommand\LINK[1]{\href{#1}{(link: \nolinkurl{#1})}}

\title[Some Stability Results for the (CGL)]{Some Stability Results for the Complex Ginzburg-Landau Equation}

\author{Sim\~ao Correia and M\'ario Figueira}

\begin{document}

\maketitle

\begin{abstract}
Using some classical methods of dynamical systems,
stability results and asymptotic decay of strong solutions for the complex Ginzburg-Landau equation (CGL),
$$ \partial_t u = (a + i\alpha) \Delta u - (b + i \beta) |u|^\sigma u + k u, \,\, t > 0,\,\, x\in \Omega, $$
with $a>0, \alpha, b, \beta, k \in \R$, are obtained.
Moreover, we show the existence of bound-states under certain conditions on the parameters and on the domain. We conclude with the proof of asymptotic stability of these bound-states when $\Omega=\mathbb{R}$ and $-k$ large enough.

\vskip10pt
\noindent\textbf{Keywords}: complex Ginzburg-Landau; stability; bound-states.
\vskip10pt
\noindent\textbf{AMS Subject Classification 2010}: 35Q56, 35B10, 35B35.
\end{abstract}

\section{Introduction and main results}\label{Intro}

This paper is concerned with the study of the stability of some equilibrium solutions of the Cauchy problem
for the complex Ginzburg-Landau equation
\begin{equation}\tag{CGL}\label{CGL}
\left\{\begin{array}{lll}
\partial_t u = (a + i\alpha) \Delta u - (b + i \beta) |u|^\sigma u + k u, & t > 0,\,\, x\in \Omega \\
u(t, x) = 0,\,\,\, x\in \partial \Omega,\,\, t\geq 0 \\
u(0, x) = u_0(x) &
\end{array}\right.
\end{equation}
where $ a >0, \alpha, b, \beta, k \in \R$ and $\Omega$ is assumed to be a domain in $\R^N$
 of class $C^2$ with $\partial \Omega$
bounded. We also analyze the asymptotic behaviour of the global solutions of \eqref{CGL}. Local and 
global existence and uniqueness of solutions of \eqref{CGL} are widely studied under several assumptions
on the parameters since the seminal paper \cite{Lever}; see also \cite{GV1},\cite{GV2},\cite{Okaza} and 
the references therein. On the other hand, there are not many results concerning the blow-up of the solutions
of \eqref{CGL}: we refer e.g. \cite{Caz-Weiss1} and \cite{Masmou}. Furthermore, concerning the existence of standing
wave solutions, some partial results were obtained in the case of a bounded connected
 domain; cf.\cite{Caz-Weiss2} and \cite{Puel}. The latter also includes a result concerning the whole space $\Omega = \R^N$.

The linear operator of \eqref{CGL}, $ -A = (a + ib)\Delta,\,\,\, a>0,\, b\in\R$, with domain,
$D(A) = H^2(\Omega) \cap H^1_0(\Omega)$, generates an analytic semi-group (see \cite{Pazy}):
writing $A_\varepsilon = A + \varepsilon I, \, \varepsilon > 0$, we have $\Re (A_\varepsilon u, u)_{L^2}
= a \| \nabla u \|_{L^2}^2 + \varepsilon \| u \|_{L^2}^2 \geq c_0 \| u \|_{H^1}^2,\, c_0 = \min \{ a, \varepsilon \} $
and $| \Im  ( A_\varepsilon u , u )_{L^2} | \leq |b| \| u \|_{H^1}^2$. Then, the resolvent set
$$ \rho (A_\varepsilon ) \supset \{ \lambda : | {\rm arg} \lambda | > \theta_\varepsilon \},\,\,
\theta_\varepsilon = \arctan (| b | /\varepsilon) < \pi/2$$
and if we choose $\varepsilon$ small enough, we conclude that $\rho (A) \supset \{ \lambda : \Re \lambda < 0 \}.$ We see also
easily that $ e^{-A z} = e^{-A_a z} e^{a z}$ is an analytic semi-group in the sector
$ \{ z : | {\rm arg} z | < \omega \},\, \omega = \pi/2 -  \arctan (|b| / a) $.

 Let introduce now the following definition:

\begin{defi}
A function $u(\cdot) \in C( [0, T); L^2(\Omega)), \,\, T>0$, is called a strong solution of \eqref{CGL} if
$u(t)\in D(A)$, ${{du}\over{dt}} (t)$ exists for $t\in(0, T),\,\, u(0) = u_0$ and the differential
equation in \eqref{CGL} is satisfied for $t\in (0, T)$.
\end{defi}
Since $f(u) = -(b + i\beta)|u|^\sigma u + k u$ is locally Lipschitz in $H^1(\Omega)$ with
values in $L^2(\Omega)$, for $\sigma \leq {2\over{N-2}}$ if $N  > 2$ and for any $\sigma > 0$
if $N=1, 2$, then there exists $T=T(u_0) > 0$ such that the problem \eqref{CGL}
has a unique solution on $[0, T_0)$, and this solution depends continuously of the
initial data (see \cite{Henry}, pag. 54 and 62). To obtain a global solution, multiply
the equation in \eqref{CGL} by $\overline {u}, \, -\Delta \overline {u}$ and
$|u|^\sigma \overline{u}$, integrate on $\Omega$ and take the real part. One obtains

 \begin{equation} \label{conserv1}
      \frac{1} {2} \frac{d} {dt} \| u \|_{L^2}^2 = - a \| \nabla u \|_{L^2}^2
 - b \| u \|_{L^{\sigma+2}}^{\sigma + 2} + k \| u \|_{L^2} 
 \end{equation}
 
 \begin{equation} \label{conserv2}
 \begin{split}
   \frac{1} {2} \frac{d} {dt} \| \nabla u \|_{L^2}^2  = -a \| \Delta u \|_{L^2}^2 
   &+ b\, \Re \int_\Omega \Delta \overline{u} | u |^\sigma u \,dx \\
   & - \beta\, \Im \int_\Omega \Delta \overline{u} | u |^\sigma u \,dx + k \| \nabla u \|_{L^2}^2
   \end{split}
   \end{equation}
   
    \begin{equation} \label{conserv3}
 \begin{split}
\frac{1} {\sigma +2} \frac{d} {dt} \| u \|_{L^{\sigma+2}}^{\sigma + 2} = 
a \,\Re  \int_\Omega \Delta u | u |^\sigma \overline{u} \,dx & - \alpha \, \Im \int_\Omega \Delta u | u |^\sigma \overline{u} \,dx  \\
& -b \| u \|_{L^{2 \sigma+2}}^{2 \sigma + 2} + k  \| u \|_{L^{\sigma+2}}^{\sigma + 2}
\end{split}
\end{equation}
Next, if we multiply \eqref{conserv3} by $\beta / \alpha$ (with $\alpha \not = 0$) and add to
\eqref{conserv1} $+$ \eqref{conserv2}, one obtains

\begin{multline*}
\frac{d}{dt} \left[ \frac {1} {2} \| u \|_{H^1}^2  + \frac {\beta} {\alpha (\sigma + 2)} \| u \|_{L^{\sigma+2}}^{\sigma + 2} \right]
= k \,[  \| u \|_{H^1}^2 + \frac {\beta} {\alpha} \| u \|_{L^{\sigma+2}}^{\sigma + 2} ]   
 -a \| \nabla u \|_{L^2}^2 \\-b \| u \|_{L^{\sigma+2}}^{\sigma + 2} 
  -a \| \Delta u \|_{L^2}^2 
- \frac {\beta b} {\alpha} \| u \|_{L^{2 \sigma+2}}^{2 \sigma + 2} 
 + \left( b + \frac{a \beta} {\alpha}\right) \Re  \int_\Omega \Delta u | u |^\sigma \overline{u} \,dx 
\end{multline*}
Since $\Re  \int_\Omega \Delta u | u |^\sigma \overline{u} \,dx \leq 0$, it is now clear that, 
if $b\geq 0$ and $\alpha \beta \geq 0$
then $\| u(t) \|_{H^1}$ is locally bounded and we can state:

\begin{prop}\label{propgeral}
Let $\Omega$ be a domain in $\R^N$ of class $C^2$ with $\partial \Omega$ bounded. Assume 
$0 < \sigma \leq \frac {2} {N-2}$ if $ N > 2$ and $ 0 < \sigma$ if $N=1, 2$. Then, for any $u_0 \in H_0^1 (\Omega )$, 
there exists $T = T(u_0) > 0$ such that \eqref{CGL} has a unique strong solution on $[0, T)$ and this
solution depends continuously of the initial data. Moreover, if $b +\alpha\beta/a\geq 0$, the solution is global.
\end{prop}

Under the conditions of the Proposition 1.2, denote by $S(t)$ the dynamical system associated to \eqref{CGL} such that
 $ S(t) u_0 \equiv u(t ; u_0),\, t\geq 0$, represents the solution of \eqref{CGL} with
initial data $u(0) = u_0 \in H_0^1 (\Omega)$, and let recall the following classical definition:
\begin{defi}
We say that $u_0 \in H_0^1(\Omega)$ is stable if for any $\delta > 0$ there exists $\varepsilon > 0$ such that
$$ v\in H_0^1 (\Omega) , \,\, \| u_0 - v \|_{H^1} < \,\varepsilon \Rightarrow \,
\sup_{t \geq 0} \| S(t) u_0 - S(t) v \|_{H^1} < \delta $$
In addition, we say that $u_0$ is asymptotically stable if $u_0$ is stable and there exists $\eta > 0$
such that $\lim_{t\rightarrow\infty} \| S(t) u_0 - S(t) y \|_{H^1} = 0$ for all $y\in H_0^1(\Omega),\,
\| u_0 - y \|_{H^1} < \eta. $
\end{defi}

The following results concern the stability of the equilibrium solution $u \equiv 0$ and the asymptotic
decay of the global solutions of \eqref{CGL} depending on the coefficient for the driving term $k$.

\begin{thm}[Stability of the zero solution for small $k$]\label{teor1}
Assume the hypothesis of the Proposition \ref{propgeral}.

\begin{enumerate}[1.]
\medskip

\item $ L^p$  stability:
\medskip
 
 \noindent If 
 $$ k \leq 0 \quad {\it and} \quad \frac{|\alpha|} {a} \leq \frac{2} {p - 2}$$
 the equilibrium point $0$ is $L^p$-stable for $2 \leq p \leq \frac {2 N} {N - 2}$, if $N > 2,\,\, 2 \leq p < \infty$ if $N=1, 2$.
 
 \noindent In addition, if $k < 0$
 $$\| u(t, x) \|_{L^p} \rightarrow 0\mbox{ as }t\rightarrow \infty, \quad \hbox{\it for all }  u_0\in H_0^1(\Omega).$$ 
 
 \noindent In the particular case $p=2$, if $\Omega$ is a bounded domain, $k > 0$ and
 $\frac {k} {a} < \left(\frac{1} {\omega_N} | \Omega | \right)^{- 2/N}$, where $\omega_N$ represents
 the volume of the unit ball in $\R^N$ and $|\Omega |$ the volume of $\Omega$, then 
 $ \| u(t, x) \|_{L^2} \rightarrow 0$ as $t\rightarrow \infty $, for all $ u_0\in H_0^1(\Omega).$
 \medskip
 
 \item $H^1$ stability:
 \medskip
 
 \noindent Let $\Omega$ a bounded domain and assume $\alpha /a = \beta / b $. Then $0$ is
 asymptotically stable in $H^1$ if
 $$
  k \leq  \frac {a} {2}  \left(\frac{1} {\omega_N} | \Omega | \right)^{- 2/N}
 $$
% \begin{enumerate}[a)]
% \item $ k \leq 0 $ or
% \item $ 0 < k \leq  \frac {a} {2}  \left(\frac{1} {\omega_N} | \Omega | \right)^{- 2/N}$
% \end{enumerate}
 \noindent In addition,
 $$ \| u(t, x) \|_{H^1} \rightarrow 0\,\mbox{ as } t\rightarrow \infty, \quad \hbox{\it for all} \ u_0\in H_0^1(\Omega).$$
 
 \end{enumerate}
 
 \end{thm}
 
% We consider now the \eqref{CGL} problem with the parameter $b < 0$. Under certain conditions,
% the next theorem prove the instability of the solution $u \equiv 0$.

 \begin{thm}[Instability of the zero solution for large $k$]\label{instability}
 Let $\Omega \subset \R^N$ be a bounded domain of class $C^2$ and denote by $\lambda_n, n=1,2, \ldots$
 the eigenvalues of $ - \Delta$ ranked in ascending order.
  Consider the \eqref{CGL} problem with $0 < \sigma \leq \frac{2} {N-2}$ if $N > 2$ and $0 < \sigma$ if $N = 1, 2$
   and assume $b < 0,\, \alpha, \beta \in \R, \, \alpha / a = \beta / b $
 and $k > 0$ such that $k/a \in (\lambda_n, \lambda_{n+1} )$ for some $n\in \N$. Then, the equilibrium
 solution $u\equiv 0$ is unstable.
 
 \end{thm}
 
 As we have noted before, the existence of standing waves for the complex Ginzbourg-Landau equation
 remains a largely open problem. Before we proceed, we rewrite the complex Ginzburg-Landau equation in its trigonometric form, following the notations of \cite{Caz-Weiss2} and \cite{Puel}:
 
  \begin{equation}\tag{CGL*}\label{GLcomplex}
 u_t = e^{i\theta} \Delta u + e^{i \gamma} | u |^\sigma u + ku
 \end{equation}
 
 \noindent where $-\pi/2 < \theta < \pi/2, -\pi<\gamma\le \pi$, $ \,k\in \R,\, u = u(t, x),\, x\in \R,\, t > 0.$
  \noindent We then look for solutions of \eqref{GLcomplex} in the form $ u = e^{i\omega t} \phi (x) $, where
 $\phi \in H^1(\R)$ is a solution of the elliptic equation
 
 \begin{equation}\label{GLelip1}\tag{B-S}
 i\omega \phi = e^{i \theta} \Delta \phi + e^{i \gamma} | \phi |^\sigma \phi + k \phi,\quad \phi\in H^1_0(\Omega).
 \end{equation}

 To stress the difficulty in finding bound-states, we start with the following simple non-existence result:
 
 \begin{prop}\label{prop:naoexistem}
 	Given $\Omega$ an open subset of $\mathbb{R}^d$, suppose that $\phi\in H^1_0(\Omega)$ is a solution of
 	$$
 	i\omega \phi = e^{i\theta}\Delta \phi + e^{i\gamma}|\phi|^\sigma \phi,\quad \omega>0, 0\le \gamma\le \theta \le \pi/2, \ \gamma\neq \pi/2.
 	$$
 	Then $\phi\equiv 0$. The same conclusion is valid if $\omega=0$ and $\theta\neq\gamma$.
 \end{prop}
\begin{rem}
	Notice that \eqref{GLelip1} may always be reduced to $k=0$: it suffices to multiply the equation by $(\omega+ik)/(\omega^2+k^2)$.
\end{rem}
% \begin{proof}
% 	Multiply the equation by $\bar{u}$ and integrate:
% 	$$
% 	\omega \|u\|_2^2 + \sin\theta \|\nabla u\|_2^2 = \sin\gamma \|u\|_{\sigma+2}^{\sigma+2},\quad \cos\gamma \|\nabla u\|_2^2 = \cos\gamma \|u\|_{\sigma+2}^{\sigma+2}.
% 	$$
% 	Then one has $\omega\cos\theta\|u\|_2^2=\sin(\gamma-\theta)\|u\|_{\sigma+2}^{\sigma+2}$. If $u\neq 0$, then either $\theta=\pi/2$ and
% 	$$
% 	\sin(\gamma-\theta)=0
% 	$$
% 	or $\theta<\pi/2$ and
% 	$$
% 	\sin(\gamma-\theta)>0.
% 	$$
% 	From the assumptions on $\gamma$ and $\theta$, both cases are impossible.
% \end{proof}
% 
% 
 As a consequence of the non-existence result, one can argue that a direct perturbative argument around the Schrödinger ground-state is impossible. In fact, the Schrödinger case corresponds to $\theta=\gamma=\pi/2$ and $\omega=1$, which lies in the frontier of the region of non-existence. If one could apply a direct perturbative argument (that is, without having a dependence between $\gamma$ and $\theta$), one would find an open region of existence around $\theta=\gamma=\pi/2$, which is impossible.
 
 Our contribution to the problem of existence of bound-states for \eqref{GLcomplex} is twofold: the first result concerns the existence on bounded domains, for certain values of $k$; the second focuses on the case $\Omega=\mathbb{R}$.
 
\begin{thm}\label{thm:existbslimitado}
	Suppose that $\Omega$ is a bounded, connected, open subset of $\mathbb{R}^d$ such that the Laplace-Dirichlet operator over $\Omega$ has a simple eigenvalue $\lambda$. Fix $0\le \theta, \gamma\le \pi/2$ and $\gamma\neq\pi/2$. There exists $\epsilon>0$ such that, for any $k$ with $0< \lambda\cos\theta - k<\epsilon$, there exist $\omega>0$ and a solution $\phi\in H^1_0(\Omega)$ of \eqref{GLelip1}
	
\end{thm}
\begin{rem}
	The above result is quite similar to that of \cite{Caz-Weiss2}: therein, it is proven that, over open bounded connected subsets of $\mathbb{R}^d$, the equation
	$$
	i\omega u=e^{i\theta}\Delta u + e^{i\gamma}|u|^\sigma u + ku
	$$
	has a solution $(\omega, u)\in \mathbb{R}\times H^1_0(\Omega)$ if $\sigma$ is sufficiently small and $(\lambda\cos\theta - k)\cos\gamma>0$. The goal of our result is to trade the freedom in $k$ for the freedom in $\sigma$.
\end{rem}

% 
% However, in the $1$-dimensional case, we can offer a collection
% of bound states and study his stability for the complex Ginzbourg-Landau equation in the 
% trigonometric form:
% \begin{equation}\tag{CGL*}\label{GLcomplex}
% u_t = e^{i\theta} \Delta u + e^{i \gamma} | u |^\sigma u + ku
% \end{equation}
% where $-\pi/2 < \theta, \gamma < \pi/2, \,k\in \R,\, u = u(t, x),\, x\in \R,\, t > 0.$
% 
% \noindent We look for solutions of \eqref{GLcomplex} in the form $ u = e^{i t} \phi (x) $ (bound-states), where
% $\phi \in H^1(\R)$ is a solution of the elliptic equation
% 
% \begin{equation}\label{GLelip1}
% i \phi = e^{i \theta} \Delta \phi + e^{i \gamma} | \phi |^\sigma \phi + k \phi
% \end{equation}
% 
\begin{thm}\label{thm:existbs}
	Consider $\Omega=\mathbb{R}$. 
	
	\noindent 1. If $\omega^2+k^2\neq 0$ and
	$
	\arg (k- i\omega )\neq \theta,
	$
	then equation \eqref{GLelip1} has at most one solution, up to gauge rotations and translations.
	
\noindent 2. Fix $-\pi/2 <\theta<\pi/2$ and $\omega, k\in \mathbb{R}$ such that $\omega\cos\theta + k\sin\theta\neq 0$. Define
	  \begin{equation}\label{eq:valor-d}
	d = \frac{ k \cos \theta - \omega\sin\theta  + \sqrt{\omega^2 + k^2}}
	{\omega\cos\theta + k \sin\theta}
	\end{equation}
	and let $\gamma\in (-\pi,\pi]$ be the unique solution of
		\begin{equation}\label{eq:restricaogamma}
	\tan(\gamma-\theta)=\frac{d(\sigma+4)}{\sigma+2-2d^2},\quad d\sin(\gamma-\theta)+\cos(\gamma-\theta)>0.
	%	 \frac{\sin (\gamma-\theta)} {d} - \cos (\gamma-\theta) = \frac{2}  {\sigma + 2}
	%	( d \sin (\gamma-\theta) + \cos (\gamma-\theta) ),
	\end{equation}
	Then
	the complex Ginzburg-Landau equation admits a bound-state of the form
	$$
	\phi=\psi\exp\left(id\ln\psi\right),
	$$
	where $\psi$ is the bound-state for the nonlinear Schrödinger equation, up to scaling and scalar multiplication. 

\end{thm}

%\begin{rem}
%	We recall that, in the case $\Omega=\mathbb{R}$, the nonlinear Schrödinger equation has a unique bound-state. It would be interesting to see if the uniqueness of bound-state still holds for the complex Ginzburg-Landau equation. To do this, it would suffice to show that all bound-states of \eqref{GLcomplex} are of the form $\psi\exp(id\psi)$.
%\end{rem}

\begin{rem}
	We observe that the conditions $\omega^2+k^2\neq 0$ and $\arg(k-i\omega)\neq\theta$ imply $\omega\cos\theta + k\sin\theta\neq 0 $. 
\end{rem}

\begin{rem}
	In the particular case $\omega=1$ and $k=0$, the condition \eqref{eq:restricaogamma} reduces to
	\begin{equation}
	\frac{\cos \gamma} {\sin ( \gamma - \theta)} =\frac {\sigma} {\sigma + 4}.
	\end{equation}
	In particular, $\gamma > \theta$, which agrees with the nonexistence result. Moreover, for $ \gamma \leq \pi/4$, the above condition
	is never verified:
	$$ \frac{\cos \gamma} {\sin (\gamma - \theta)} > \frac{ \cos \gamma} {\sin \gamma} \geq 1 > 
	\frac{\sigma}  {\sigma + 4}.$$
\end{rem}

\noindent Finally, we present a stability result for bound-states with large $-k$:

\begin{thm}\label{thm:estabilidadebs}
	Fix $\omega=1$ and $0<\theta<\pi/2$. 
	Then, for $-k$ sufficiently large, the bound-state $\phi$ built in Theorem \ref{thm:existbs} is asymptotically stable.
\end{thm}

%
%\begin{thm}\label{bound-stat} \medskip
%
%
%\begin{enumerate}[1.]
%\item Assume $k=0$ in the \eqref{GLcomplex} equation. Let $\theta, \gamma$ and $\sigma$ with
%$ -\pi/2 < \theta, \gamma < \pi/2,\, \sigma > 0$, verifying the following relationship:
%$$ \frac {\cos \gamma} {\sin (\gamma - \theta)} > \frac {\sigma} {\sigma + 4}. $$
%Then, there exists a positive function $a_0 \in H^1 (\R)$ such that, 
%$$ u(t, x) = e^{i t} a_0 e^{i d_0 \ln a_0}, \quad \hbox{\it with} \quad  d_0 = \frac {1 - \sin \theta} {\cos \theta}, $$
%is a bound-state of \eqref{GLcomplex}.
%\medskip
%
%\item Assume $k < 0.$ Then, for any $\sigma > 0$, and $\theta\in (0, \pi/2)$ there exists $M > 0$ 
%such that for $-k > M$ 
%there exists $\gamma \in (-\pi/2, \pi/2)$ and a positive function $a \in H^1(\R)$ such that
%$$ u(t, x) = e^{i t} a\, e^{i d \ln a}, \quad \hbox{\it with} \quad  
%d = \frac {k \cos \theta -  \sin \theta + \sqrt{1 + k^2}} {\cos \theta + k \sin \theta}, $$
%is a bound-state of \eqref{GLcomplex}, asymptotically stable.
%
%
%\end{enumerate}
%\end{thm}
 
\vskip10pt
\section{Proof of the Theorem 1.4}
 
\vskip10pt

Let $ S(t)$ be a dynamical system on a Banach space $H$  and recall that a Lyapunov
function is a continuous function $W : H \rightarrow \R$ such that
$$ \dot W (u) := \limsup_{t\rightarrow 0^+} \, \frac{1} {t} [ W (S(t) u) - W(u) ] \leq 0$$
for all $u\in H$. 
The next lemma is mainly proved in \cite{Hale}.

\begin{lem}\label{lemaHale}
Let $S(t)$ be a dynamical system on a Banach space $(D, \| \, \|)$. Let $E$ a normed space
such that $D\hookrightarrow E$ and $W$ a Lyapunov function on $D$ such that
$$ W(u_0) \geq k_1 \| u_0 \|_E,\,\, k_1 > 0,\, u_0 \in D. $$
Then, the equilibrium point $0$ is $\|\,\|_E$ - stable in the sense that
$$ u_0 \in D,\,\,  \| u_0 \|\rightarrow 0 \Rightarrow \| S(t) u_0 \|_E \rightarrow 0, $$
 uniformly in $t \geq 0$.

\noindent Assume in addition that
$$ \dot W (u_0) \leq - k_2 \| u_0 \|_E,\,\, k_2 > 0, \, u_0 \in D.$$
Then, $lim_{t\rightarrow \infty} \| S(t) u_0 \|_E = 0$ for any $u_0 \in D$.

\end{lem} 
\medskip

%\noindent {\it Proof of the Theorem 1.4}.  
\begin{proof}[Proof of Theorem 1.4]
 1. Let denote by $S(t)$ the dynamical system associated
to \eqref{CGL}, where $S(t) u_0 \equiv u(t, u_0)$ represents the unique global solution of \eqref{CGL} 
under the hypothesis of the Proposition \eqref{propgeral} and define
$$ W_p (u) = \int_\Omega | u(x) |^p dx,$$
with $ 2 \leq p \leq \frac{2 N} {N - 2} $ if $N > 2$, $ 2 \leq p < \infty$ if $N =1, 2$
and $ u = u(t, u_0)$. It is clear that $W_p : H_0^1 (\Omega) \rightarrow \R$, is a continuous functional and,
from  $\dot W_p (u) = \nabla W(u) \cdot \frac {d} {d t} u(t)$, we get
\begin{equation}
\begin{split}
\dot W_p(u)  = &\, p \,\Re \int_\Omega | u |^{p-2} \overline{u} \left\{ (a + i \alpha) \Delta u -
 (b + i \beta) | u |^\sigma u + ku \right\} dx \\
  \le & \,p k \int_\Omega | u |^p dx - ap \int_\Omega | u |^{p-2} | \nabla u|^2 dx \\
 & - p b \int_\Omega | u |^p | u |^\sigma dx + p \alpha \Im \int_\Omega \nabla \left(| u |^{p-2}\right) \overline{u} \nabla u\, dx. \\
\end{split}
\end{equation}
Since
$$ \nabla | u |^{p-2} = \frac{p-2} {2} | u |^{p-4} ( u \nabla \overline{u} + \overline{u} \nabla u) $$
we obtain
$$ \left| p \alpha \Im \int_\Omega \nabla \left( | u |^{p-2}\right) \overline {u} \nabla u dx \right| \leq
p | \alpha | \frac {p-2} {2} \int_\Omega | u |^{p-2} | \nabla u |^2 dx. $$
Hence, if $ \alpha \frac{p-2} {2} \leq a $ and $ k \leq 0$ we derive that $ \dot W_p (u) \leq p k \| u \|_{L^p}^p$
and the conclusion follows from the Lemma 2.1. If $p = 2$ and $\Omega$ is bounded,
the above estimation reduces to
$$ \dot W_2 (u) = - 2 a \int_\Omega | \nabla u |^2 dx - 2 b \int_\Omega | u |^{\sigma +2} dx
+ 2 k \int_\Omega | u |^2 dx $$
and if $k \left( \frac {1} {\omega_N} | \Omega | \right)^{2/N} < a$,
 by Poincar\'e inequality, the conclusion remains valid.
 
 \medskip
 \noindent 2. We now consider $\Omega$ a bounded domain and we define a new functional:
 \begin{equation}\label{Vliap}
 V(u) := \frac {a} {2} \int_\Omega | \nabla u|^2 dx + \frac {b} {\sigma +2} \int_\Omega | u |^{\sigma + 2} dx
 - \frac {k} {2} \int_\Omega | u |^2 dx. 
 \end{equation}
 $V$ is a continuous real function on $H_0^1(\Omega)$ and, if 
 $$k \leq 0 \quad  {\rm or}\quad k  > 0,\,\, \frac {k} {a} 
 < \left( \frac {1} {\omega_N} | \Omega | \right)^{2/N} $$
  we have $ V(u) \geq c \| u \|_{H^1(\Omega)},\, c > 0$.
   In addition, for any $u\in H_0^1(\Omega) \cap H^2 (\Omega)$ and $ h \in H_0^1(\Omega)$
  we have $ V(u + h) = V(u) + L \cdot h + o (\| h \|_{H^1}) $, where
  $$ L \cdot h = - \Re \int_\Omega \left[ a \Delta \overline{u} - b | u |^\sigma \overline{u} 
  + k \overline{u} \right]  h \, dx. $$
  Therefore, for all $u = u(t) \in H_0^1(\Omega) \cap H^2 (\Omega)$ we get
  \begin{equation}
  \begin{split}
  \dot V(u) = & - \int_\Omega  \left| a \Delta u - b | u |^\sigma u + k u \right|^2  dx \\
  & - \Re \int_\Omega \left( a \Delta u - b | u |^\sigma u \right) i \left( \alpha \Delta u - \beta | u |^\sigma u \right) dx \\
  & - \Re \int_\Omega i \,k \overline{u}  \left( \alpha \Delta u - \beta | u |^\sigma u \right) dx \\
  \end{split}
  \end{equation}
  and, for $\frac {\alpha} {a} = \frac {\beta} {b} = c,\, c\in \R$, we obtain
  \begin{equation}\label{Vliap2} 
  \dot V (u(t)) = - \int_\Omega \left| a \Delta u - b | u |^\sigma u + k u \right|^2  dx \leq 0,\,\, t > 0.
  \end{equation}
  Note that
  $$ \frac {1} {t} \,[V(S(t) u_0) - V(u_0)] = \dot V (S (t^*) u_0)$$
  for some $0 < t^* < t$ and so \eqref{Vliap2} is true for all $t \geq 0$. Hence, the functional $V$ is a
  Lyapunov function and we have the stability in $H_0^1(\Omega)$ of the equilibrium solution $ u \equiv 0$.
  
  \noindent We prove now the asymptotic stability. First, we remark that
  \begin {equation}\label{Vliap3}
  \begin{split}
  - \dot V (u) = & a^2 \int_\Omega | \Delta u|^2 dx + b^2 \int_\Omega | u |^{2 \sigma + 2} dx + k^2 \int_\Omega | u |^2 dx \\
  & - 2 a b \,\Re \int_\Omega \Delta u | u |^\sigma \overline{u} dx + 2 a k \, \Re \int_\Omega \Delta u \overline{u} dx \\
  & - 2 b k\, \Re \int_\Omega | u |^{\sigma + 2} dx. \\
  \end{split}
  \end{equation}
  The first, fourth, fifth and sixth terms in the right hand side can be estimate as follows :
  $$ \int_\Omega | \nabla u|^2 dx \leq \left| \int_\Omega \Delta u \overline{u} \,dx \right| \leq
  \left( \int_\Omega | \Delta u |^2 dx \right)^{1/2} \left( \int_\Omega | u |^2 dx \right)^{1/2}
  $$
  and by the Poincar\'e innequality,
  $$ \left( \int_\Omega | \Delta u |^2 dx \right)^{1/2} \geq \left(\frac {1} {\omega_N} | \Omega | \right)^{- 1/N}
  \left( \int_\Omega | \nabla u |^2 dx \right)^{1/2}. $$
  It follows that
  \begin{equation}\label{Vliap4}
  a^2 \int_\Omega | \Delta u |^2 dx \geq a^2 \left(\frac {1} {\omega_N} | \Omega | \right)^{- 2/N}
   \int_\Omega | \nabla u |^2 dx.
   \end{equation}
   Next
   \begin{equation}
   \begin{split}
   \Re \int_\Omega \Delta u | u |^\sigma \overline{u} dx  & = -\int_\Omega | \nabla u |^2 | u |^\sigma dx \\
   & - \frac {\sigma} {2} \Re \int_\Omega | u |^{\sigma -2} \nabla u \cdot ( \nabla u \overline{u}
   + u \nabla \overline{u} ) \overline{u} \, dx \\
   & = - \int_\Omega | \nabla u |^2 | u |^\sigma dx - \frac {\sigma} {2} \int_\Omega | \nabla u |^2 | u |^\sigma dx \\
   & \quad - \frac {\sigma} {2} \Re \int_\Omega | u |^{\sigma -2} (\nabla u \cdot  \nabla u ) \overline{u}^2 dx  \\
   & \leq - \int_\Omega | \nabla u |^2 | u |^\sigma dx \\
   \end{split}
   \end{equation}
   and so
   \begin{equation}\label{Vliap5}
   - 2 a b \,\Re \int_\Omega \Delta u | u |^\sigma \overline{u} dx \geq 2 a b \int_\Omega | \nabla u |^2 | u |^\sigma dx.
   \end{equation}
   Also
   \begin{equation}\label{Vliap6}
   2 a k\, \Re \int_\Omega \Delta u \overline{u} dx = - 2 a k \int_\Omega | \nabla u |^2 dx. 
   \end{equation}
   Finally
   \begin{itemize}
   \item[$i)$] If $k \leq 0 $, then
   $$ - \dot V (u) \geq c \int_\Omega | \nabla u |^2 dx,\quad c = a^2 \left(\frac {1} {\omega_N} | \Omega | \right)^{- 2/N} $$
   \item[$ii)$]  If $ 0 < k < \frac {a}  {2} \left(\frac {1} {\omega_N} | \Omega | \right)^{- 2/N} $
   \end{itemize}
   it follows from \eqref{Vliap3}, \eqref{Vliap4} and \eqref{Vliap5} that
   $$ - \dot V (u) \geq \delta \int_\Omega | \nabla u |^2 dx $$
   with $ \delta = a^{2} \left(\frac {1} {\omega_N} | \Omega | \right)^{- 2/N}  - 2 a k > 0$ and we obtain,
   by the Lemma 2.1, in any case,
   $i)$ or $ii)$, the asymptotic stability and the decay $\| u(t, u_0) \|_{H^1} \rightarrow 0,$ as $ t\rightarrow \infty$,
   for any $u_0 \in H_0^1( \Omega)$.
   \end{proof}
 
\vskip10pt
   \section{Proof of the Theorem 1.5}
    
   \vskip10pt
  Consider the \eqref{CGL} problem with $a, k > 0,\, b< 0, \,\alpha, \beta \in \R$ and $\Omega$ a bounded 
  domain of class $C^2$. Assume that $ 0 < \sigma \leq \frac{2} {N - 2}$ if $N > 2$, $0 < \sigma $ if $N = 1, 2$
   and $ \alpha /a = \beta /b$.
  As we pointed out before, for any $u_0\in H_0^1(\Omega)$ there exists a maximal solution of \eqref{CGL}
  defined on $[0, T), \,T > 0$, denoted by $S(t) u_0 \equiv u(t, u_0)$.
   Suppose that $\sup_{t\geq 0} \| S(t) u_0 \|_{H^1} \leq M$ for some $M > 0$. Then $\{ S(t) u_0 \}_{t\geq 0}$
   is a relatively compact set in $H_0^1(\Omega)$ (this is a consequence of the fact that the operator
   $ (a + i \alpha ) \Delta $ has compact resolvent (see e.g. \cite{Henry}, pag.57)) and therefore, it follows
   that the $\omega$ - limit set of $u_0$,
   $$ \omega (u_0) = \{ y\in H_0^1 (\Omega) : \hbox{\rm there exists}\,\,  t_n \rightarrow \infty \,\,
   \hbox{such that}\,\, S(t_n) u_0 \rightarrow y \} ,$$
   is a nonempty, invariant and compact set (see e.g. \cite{Haraux}).
   
   \noindent Consider once again the Lyapunov functional
   $$ V(u) :=  \frac{1}{2}\int_\Omega \left\{ a | \nabla u|^2  + \frac {2 b} {\sigma +2}  | u |^{\sigma + 2} 
 - k  | u |^2 \right\} dx $$
 and note (see \eqref{Vliap2}) that for any $u\in H_0^1 (\Omega),\, \dot V (u) = 0$ if and only if
 \begin{equation}\label{Dirich}
 - \Delta u - \frac{k} {a} u  =  - \frac {b} {a} | u |^\sigma u 
 \end{equation}
 \medskip
 
 \begin{proof}[Proof of Theorem 1.5.] Step 1. 
 First of all we show that there exists a small enough positive number, $\eta > 0$, such that do not
 exists a nontrivial solution $u \in H_0^1 (\Omega)$ of \eqref{Dirich} with $ \| u \|_{H^1}  =  \eta$.
 Let  $u \in H_0^1 (\Omega),\, \| u \|_{H^1} = \eta $ be a solution of \eqref{Dirich}. Since $k/a \in \rho (- \Delta )$
 and $ | u |^\sigma u \in L^2 (\Omega )$ it follows that $ u \in H^2 (\Omega ) \cap H_0^1 (\Omega) $ and
 $$ \| u \|_{H^2} \leq c\, \| u \|_{H^1}^{\sigma +1} \leq  c \, \eta^{\sigma + 1} $$
 which implies $u \equiv 0$ if $\eta $ is small enough.
 
  \medskip
 \noindent Step 2. Consider the set
 $$ K = \{ \phi \in H_0^1 (\Omega) : V (\phi ) < 0 \} $$
 and let be $k/a \in (\lambda_n, \lambda_{n+1})$ for some $n = 1, 2, \cdots , $
 with $\lambda_n $ the eigenvalues of $- \Delta$. Then, for any $\varepsilon > 0$, 
 $K \cap U_\varepsilon (0) \not = \emptyset$, where $U_\varepsilon (0) $ is the $\varepsilon$ - neighbourhood
 of $0$ in $H_0^1 (\Omega)$. To prove that, take $k/a = \lambda_n + \delta,\,\, 0 < \delta < \lambda_{n+1}
 - \lambda_n$ and consider $u_n$ the eigenfunction associated to $\lambda_n$ with $\| u_n \|_{H^1} < \varepsilon$.
 Since $ - \Delta u_n = \lambda_n u_n$ we obtain
  \begin{multline*}
a \int_\Omega  | \nabla u_n |^2 dx - k \int_\Omega | u_n |^2  dx  =  a \left[ \int_\Omega | \nabla u_n |^2 dx
 - \frac {k}{a} \int_\Omega | u_n |^2  dx \right] \\
  = a \left[  \int_\Omega | \nabla u_n |^2 dx - \lambda_n \int_\Omega  | u_n |^2 dx
   - \delta \int_\Omega | u |^2  dx \right]  < 0 
 \end{multline*}
 and so $V (u_n) < 0$. 
 
 \medskip
 \noindent Step 3. Let $U$ be an open neighbourhood of $0$ in $ H_0^1(\Omega )$ such  that
 $ U \subset \overline{U} \subset U_\eta (0)$, where $ \eta > 0$ was determined before in the step 1. and 
 let be $ \phi \in K \cap U$. We claim that $S(t) \phi \in \partial \,\overline{U}$ for some $ t > 0$.
 
 \noindent Suppose that we have $\{ S(t) \phi \} \subset U$ for all $t > 0$. Since $\overline{\{ S(t) \phi \}}$ is
 a compact set and $V (S(t) \phi )$ is a nonincreasing function of $t \geq 0$, it turns out
 that $ l = \lim_{t\rightarrow \infty} V (S(t) \phi )$ exists in $\R$. On the other hand, by the precompacity
 of $\{ S(t) \phi \}_{t\geq 0}$, the $\omega$ - limit set $\omega (\phi)$ is nonempty. Let be
 $y \in \omega (\phi) \subset U_\eta (0)$; it is clear that $ V(y) = l$ and, by the invariance of
 $\omega (\phi)$, $S(t) y \in \omega (\phi)$ and so $ V (S(t) y) = l$. It follows that
 $\dot V (y) = 0$ and by the step 1., $y = 0$. Therefore, we conclude that $\omega (\phi) = \{ 0 \}$.
 But, for any $z \in \omega (\phi),\, z = \lim_{t_n\rightarrow \infty} S(t_n) \phi$,
 $$ V (z) = \lim_{t_n\rightarrow \infty} V (S(t_n) \phi ) \leq V (\phi) < 0$$
 which is absurd. We have proved that $S(t) \phi$ must reach $ \partial \, \overline{U}$ for
 some $t > 0$ which means the instability of the equilibrium solution $ u \equiv 0$. 
 \end{proof}
 
\vskip10pt
 \section{Existence and stability of bound-states of \eqref{GLcomplex}}
 \vskip10pt
 
 \begin{proof}[Proof of Proposition \ref{prop:naoexistem}]
 	Multiply the equation by $\bar{\phi}$ and integrate:
 	$$
 	\omega \|\phi\|_2^2 + \sin\theta \|\nabla \phi\|_2^2 = \sin\gamma \|\phi\|_{\sigma+2}^{\sigma+2},\quad \cos\gamma \|\nabla \phi\|_2^2 = \cos\gamma \|\phi\|_{\sigma+2}^{\sigma+2}.
 	$$
 	Then one has $\omega\cos\theta\|\phi\|_2^2=\sin(\gamma-\theta)\|\phi\|_{\sigma+2}^{\sigma+2}$. If $\phi\not\equiv 0$, then either $\theta=\pi/2$ and
 	$$
 	\sin(\gamma-\theta)=0
 	$$
 	or $\theta<\pi/2$ and
 	$$
 	\sin(\gamma-\theta)>0.
 	$$
 	From the assumptions on $\gamma$ and $\theta$, both cases are impossible.
 	\end{proof}
 
 We now focus on the proof of Theorem \ref{thm:existbslimitado}. 
 Define $H:=\{u\in H^1_0(\Omega): \Delta u\in L^2(\Omega) \}$. This space is a real Hilbert space when equipped with the scalar product
 $$
 (u,v)=\Re \int_{\Omega} u\bar{v} + \Re \int_\Omega \Delta u\Delta \bar{v}.
 $$
 
 Let $\lambda$ be an eigenvalue of $-\Delta: H\mapsto L^2(\Omega)$. Assuming that the corresponding eigenspace is of the form $\mathbb{C} \phi$, we set
 $$
 H=\mathbb{C} \phi \oplus H_1, \quad H_1=(\mathbb{C} \phi)^\perp.
 $$
 
 \begin{lem}
 	Suppose that $\Omega$ is a bounded, connected, open subset of $\mathbb{R}^d$ such that the correspoding Laplace-Dirichlet operator has a simple eigenvalue $\lambda$. Fix $\theta, \gamma\in \mathbb{R}$ and $\sigma>0$. Then there exist $\mu_0>0$ and $C^1$ mappings
 	$$
 	v:(-\mu_0,\mu_0)\mapsto H, \quad \omega, k:(-\mu_0,\mu_0)\mapsto \mathbb{R}
 	$$
 	such that $v(0)=\phi$, $\omega(0)=-\lambda\sin\theta$, $k(0)=\lambda\cos\theta$ and
 	$$
 	\Delta v + \mu e^{i(\gamma-\theta)}|v|^\sigma v + (k-i\omega)e^{-i\theta}v=0.
 	$$ 
 \end{lem}
 
 \begin{proof}
 	Define the mapping $F:\mathbb{R}\times H_1\times \mathbb{R}\times \mathbb{R}\mapsto L^2(\Omega)$ as
 	\begin{equation}\label{eq:definicaoF}
 	F(\mu, \zeta, \omega, k)=\Delta v + \mu e^{i(\gamma-\theta)}|v|^\sigma v + (k-i\omega)e^{-i\theta}v, \quad v=\phi+\zeta.
 	\end{equation}
 	If one sets
 	$$
 	k_0-i\omega_0 = \lambda e^{i\theta}, 
 	$$
 	it follows that $F(0,0,\omega_0, k_0)=0$. Furthermore, the mapping $(\zeta, \omega, k)\mapsto F(\mu, \zeta, \omega, k)$ is of class $C^1$ and
 	$$
 	\frac{\partial F}{\partial \zeta}(\mu,\zeta,\omega, k) w=\Delta w + \mu e^{i(\gamma-\theta)}\left(|v|^\sigma w + \sigma |v|^{\sigma-2}v\Re (\bar{v}w)\right)  + (k-i\omega)e^{-i\theta}w,
 	$$
 	$$
 	\frac{\partial F}{\partial \omega}(\mu,\zeta,\omega, k)= -ie^{-i\theta} v,\quad \frac{\partial F}{\partial k}(\mu,\zeta,\omega, k)= e^{-i\theta} v.
 	$$
 	Now we check that the jacobian
 	$$
 	J=\frac{\partial F}{\partial (\zeta, \omega, k)}(0,0,\omega_0,k_0): H_1\times \mathbb{R}\times\mathbb{R} \mapsto L^2(\Omega)
 	$$
 	is a bijection. Applying to an element $(w,y,z)\in H_1\times \mathbb{R}\times\mathbb{R}$, we have
 	$$
 	J(w,y,z)=\Delta w + \lambda w + e^{-i\theta}(z-iy) \phi.
 	$$
 	If $J(w,y,z)=0$, then
 	$$
 	0=\int J(w,y,z)\bar{\phi} = e^{-i\theta}(z-iy)\|\phi\|_2^2,
 	$$
 	which implies $y,z=0$. Thus $-\Delta w=\lambda w$ and so $w$ is an eigenvector with eigenvalue $\lambda$. However, since $w\in H_1$, this means that $w=0$. Hence $J$ is injective. On the other hand, given $f\in L^2(\Omega)$, write
 	$$
 	f=-e^{-i\theta}(\tilde{z}-i\tilde{y})\phi + \psi, \quad \int_\Omega \psi\bar{\phi}=0.
 	$$
 	The orthogonality condition implies that there exists $\tilde{w}\in H_1$ such that $\Delta \tilde{w} + \lambda \tilde{w} = \psi$. Then
 	$$
 	J(\tilde{w},\tilde{y},\tilde{z})= f,
 	$$
 	which shows that $J$ is surjective.
 	
 	With the above considerations, one may apply the Implicit Function Theorem \cite[Theorem 4.B]{zeidler} and the proof is finished.
 \end{proof}
 
 \begin{proof}[Proof of Theorem \ref{thm:existbslimitado}]
 	Consider the mappings $v,\omega, k$ from the previous theorem and the mapping $F$ as in \eqref{eq:definicaoF}. Then
 	$$
 	F(\mu, \zeta(\mu), \omega(\mu), k(\mu)) = 0, \quad \mu\in (-\mu_0,\mu_0).
 	$$
 	Differentiating with respect to $\mu$ at $\mu=0$, we obtain
 	$$
 	e^{i(\gamma-\theta)}|\phi|^\sigma \phi + \Delta w + \lambda w + e^{-i\theta}\left(\frac{\partial k}{\partial \mu} - i\frac{\partial \omega}{\partial \mu}\right)\phi = 0,\quad w=\frac{\partial v}{\partial \mu}.
 	$$
 	Multiplying by $\bar{\phi}$ and integrating over $\Omega$, we arrive at
 	$$
 	\frac{\partial k}{\partial \mu}=-\cos\gamma\frac{\|\phi\|_{\sigma+2}^{\sigma+2}}{\|\phi\|_2^2}<0,\quad \frac{\partial \omega}{\partial \mu}=\sin\gamma\frac{\|\phi\|_{\sigma+2}^{\sigma+2}}{\|\phi\|_2^2}.
 	$$
 	Thus the mapping $\mu\mapsto k(\mu)$ is locally invertible at 0, which implies that one may write $\mu=\mu(k)$, $v=v(k)$ and $\omega=\omega(k)$, for $|k-\lambda\cos\theta|<\epsilon$. Finally, if $\lambda\cos\theta-k>0$, then $\mu(k)>0$ and so $u(k)=\mu(k)^{\frac{1}{\sigma}}v(k)$ satisfies \eqref{GLelip1}.
 \end{proof}

 \begin{proof}[Proof of Theorem \ref{thm:existbs}]
 We start with the uniqueness statement. We write the equation as a system of ODE's and linearize around the trivial solution. One then checks that the linear system has two eigenvalues $\pm\lambda$, with $\Re \lambda>0$ and that each eigenvalue has a two-dimensional eigenspace. It follows that stable and unstable manifolds ($\mathcal{S}$ and $\mathcal{U}$) for the full equation \eqref{GLelip1} have dimension two.
 
 Given a bound-state $\phi\in \mathcal{U}\cap\mathcal{S}$, define
 $$
 \mathcal{U}_0:=\{ e^{i\rho}(\phi(x),\phi'(x))\ :\ \rho\in \mathbb{T}, x\in \mathbb{R}  \}.
 $$
 We claim that, if for some $(\rho_0,x_0)\neq (\rho_1, x_1) $, 
 $$
 e^{i\rho_0}(\phi(x_0),\phi'(x_0))=e^{i\rho_1}(\phi(x_1),\phi'(x_1)),
 $$
 then $\phi$ would not go to $0$ as $t\to+\infty$.  Indeed, either $x_1=x_0$ (which implies the contradiction $\rho_1=\rho_2$) or $x_1\neq x_0$. If $x_1>x_0$, then the gauge and translation invariances of the equation imply that, for $\delta=x_1-x_0$,
 $$
 (\phi(x_0+n\delta),\phi'(x_0+n\delta))=e^{in(\rho_0-\rho_1)}(\phi(x_0),\phi'(x_0)),\quad n\in\mathbb{N}
 $$
 which is again a contradiction when $n\to\infty$. Hence $\mathcal{U}_0$ is a two-dimensional submanifold of $\mathcal{U}$. 
 
 Next, we show that $\mathcal{U}_0$ is closed in $\mathcal{U}$: given $(y,z)\in \mathcal{U}\cap\overline{\mathcal{U}_0}$, then
 $$
 e^{i\rho_n}(\phi(x_n),\phi'(x_n))\to (y,z), \quad (\rho_n, x_n)\in \mathbb{T}\times\mathbb{R}.
 $$
 If $(\rho_n, x_n),\ n\in \mathbb{N}$, remains in a bounded set, then, up to a subsequence, $(\rho_n, x_n)\to (\rho_0,x_0)$ and so
 $$
 (y,z)=e^{i\rho_0}(\phi(x_0),\phi'(x_0))\in \mathcal{U}_0.
 $$
 If not, then there exists a subsequence $(x_{n_j})_{j\in\mathbb{N}}$ such that $x_{n_j}\to \pm \infty$. But then $(y,z)=(0,0)\notin \mathcal{U}$, which is absurd.
 We conclude that $\mathcal{U}_0$ is closed in $\mathcal{U}$
 
 From the invariance of domains, $\mathcal{U}_0$ is open in $\mathcal{U}$. Since $\mathcal{U}$ is connected, $\mathcal{U}_0=\mathcal{U}$ and so, up to gauge rotations, there can only be one solution on the unstable manifold. Since  any solution of \eqref{GLelip1} must lie on $\mathcal{U}$, one obtains uniqueness of bound-states.
 
 Now we focus on the existence statement. We look for solutions $\phi \in H^1(\R)$
 	of \eqref{GLelip1} or, in an equivalent way,
 	\begin{equation}\label{equ-est2}
 	\phi'' = \omega e^{i \tilde \theta} \phi -  e^{i \tilde \gamma} | \phi |^\sigma \phi + i k\,  e^{i \tilde \theta} \phi
 	\end{equation}
 	with $ \tilde \theta = \pi/2 - \theta,\, \tilde \gamma = \gamma - \theta$.
 	
 	\noindent Let us search a solution of the equation \eqref{equ-est2}, $ \phi \in H^1(\R)$, of the form
 	\begin{equation}\label{sol}
 	\phi = \psi\, \exp (i \,d\, \ln \psi) 
 	\end{equation}
 	where $d\in \R$ and $\psi > 0$ is the unique solution of the stationary Schr\"odinger equation
 	\begin{equation}\label{Schr}
 	\psi'' = \epsilon \,\psi - \eta \,\psi^{\sigma + 1}, \quad \epsilon, \eta > 0. 
 	\end{equation}
 	First, one has
 	\begin{equation}\label{second-deriv}
 	\phi'' ( x ) = \left[ \psi'' (x) (1 + i d) + id (1 + i d) \frac{\psi'(x)^2} {\psi(x)} \right] \exp ( i d \ln \psi(x) ).
 	\end{equation}
 	Next, we note that if $\psi$ is a solution of \eqref{Schr}, then a direct integration of the equation yields
 	\begin{equation}\label{Schr-new}
 	\frac{(\psi')^2} { \psi} = \epsilon\, \psi - \frac {2 \eta} {\sigma + 2} \,\psi^{\sigma + 1}.
 	\end{equation}
 	It follows from \eqref{equ-est2} that
 	\begin{align*}
 		\psi'' - d^2 \frac{(\psi')^2} {a} & = \omega\cos \tilde{\theta} \,\psi - k \sin \tilde{\theta} \,\psi - \cos \tilde{\gamma} \,\psi^{\sigma +1}, \\
 		d \psi'' + d \frac{(\psi')^2} {\psi} & = \omega\sin \tilde{\theta} \, \psi + k \cos \tilde{\theta} \,\psi - \sin \tilde{\gamma} \,\psi^{\sigma +1} 
 	\end{align*}
 	and so
 	
 	\begin{multline*}
 		( 1 + d^2) \psi''   = [\omega(d \sin \tilde{\theta} + \cos \tilde{\theta}) 
 		+ k (  d \cos \tilde{\theta} - \sin \tilde{\theta})] \,\psi  \\
 		- ( d \sin \tilde{\gamma} + \cos \tilde{\gamma} ) \,\psi^{\sigma +1} ,
 	\end{multline*}
 	
 	\begin{multline*}
 		(1 + d^2) \frac{(\psi')^2} {\psi}  = \biggl[ \omega\biggl(\frac{\sin \tilde{\theta}} {d} - \cos \tilde{\theta}\biggl) \psi
 		+ k  \biggl(\frac{\cos \tilde{\theta}} {d} + \sin \tilde{\theta}\biggl) \psi \\
 		- \biggl( \frac{\sin \tilde{\gamma}} {d} - \cos \tilde{\gamma} \biggl) \psi^{\sigma +1} \biggl] .
 	\end{multline*}
 	Hence, writing
 	\begin{equation}\label{alfa}
 	\epsilon = \frac { \omega(d \sin \tilde{\theta} + \cos \tilde{\theta}) + k ( d \cos \tilde{\theta}  - \sin \tilde{\theta}) }
 	{ 1 + d^2}
 	\end{equation}
 	and
 	\begin{equation}\label{beta}
 	\eta = \frac { d \sin \tilde{\gamma} + \cos \tilde{\gamma} }
 	{ 1 + d^2}
 	\end{equation}
 	we require that
 	\begin{equation}\label{equ-final1}
 	\omega\biggl( \frac{\sin \tilde{\theta}} {d} - \cos \tilde{\theta}\biggl) + k \biggl( \frac{\cos \tilde{\theta}} {d} + \sin \tilde{\theta} \biggl) 
 	= \omega(d \sin \tilde{\theta} + \cos \tilde{\theta} )
 	+ k ( d \cos \tilde{\theta} - \sin \tilde{\theta})
 	\end{equation}
 	and
 	
 	\begin{equation}\label{equ-final2}
 	\frac{\sin \tilde{\gamma}} {d} - \cos \tilde{\gamma} = \frac{2}  {\sigma + 2}
 	( d \sin \tilde{\gamma} + \cos \tilde{\gamma} ) 
 	\end{equation}
 	From \eqref{equ-final1} we derive
 	\begin{equation}\label{valor-d}
 	d = \frac{ k \sin \tilde{\theta} - \omega\cos \tilde{\theta}  \pm \sqrt{\omega^2 + k^2}}
 	{\omega\sin \tilde{\theta} + k \cos \tilde{\theta}} =: d_{\pm}
 	\end{equation}
 	and so
 	$$
 	\epsilon=\pm\sqrt{\omega^2+k^2}.
 	$$
 	However, if $d = d_-$ then $\psi$ would be a bound state of a nonlinear Schr\"odinger equation
 	with negative frequency, $\epsilon$, which does not exists (see \cite{Caz}). Thus, we must have $d = d_+$. Finally, the definition  of $\gamma$ (cf. \eqref{eq:restricaogamma}) is equivalent to \eqref{equ-final2} and $\eta>0$. 
 	\end{proof}
 
 Finally, we turn our attention to the proof of Theorem \ref{thm:estabilidadebs}.
% We look for solutions of \eqref{GLcomplex} of the form $ u(t, x) = e^{i \omega t} \phi (x)$, and so
% \begin{equation}\label{equ-est}
% i \omega \phi = e^{i \theta} \Delta \phi + e^{i \gamma} | \phi |^\sigma \phi + k \phi ,
% \end{equation}
% with $ \pi/2 < \theta , \gamma < \pi/2,\, k\in \R , \sigma > 0$. In the following, we consider the 
% one dimensional case and, for simplicity, we take $\omega = 1$. We start to notice that the 
% solution $\phi \in H^1 (\R)$ of the elliptic equation \eqref{equ-est} is an equilibrium point of the equation
% \begin{equation}\label{equ-equilib}
% u_t = e^{i \theta} u'' + e^{i \gamma} | u |^\sigma u  - i u + k u .
% \end{equation}
We start to notice that 
a bound-state $\phi \in H^1 (\R)$ of \eqref{GLcomplex} is an equilibrium point of the equation
 \begin{equation}\label{equ-equilib}
 u_t = e^{i \theta} u'' + e^{i \gamma} | u |^\sigma u  - i\omega u + k u .
 \end{equation}
 Linearizing around $\phi$, we obtain
  \begin{equation}\label{equ-linear}
 u_t = e^{i \theta} u'' + e^{i \gamma} | \phi |^\sigma u  + 
 \sigma e^{i \gamma} | \phi |^{\sigma -2} \phi \,\Re \,(\overline{\phi} \,u) - i\omega u + k u \equiv - L u .
 \end{equation}
 If the spectrum $ \sigma (L)$  is in $\{ \Re \, \lambda > \eta \}$ for some $\eta > 0$, it is
 well known that the equilibrium point $\phi$ of \eqref{equ-linear} is asymptotically stable
 and thus we need some information about the spectrum of the operator $L$.
 
 \begin{defi} Let $L$ be a linear operator in a Banach space. We call a {\it normal point} any element of
 the resolvent set or an isolated eigenvalue of $L$ of finite multiplicity. The set of the normal points
 of $L$ is represented by $\tilde{\rho} (L)$. We define the essencial spectrum of $L$ as the
 set $\sigma_e  (L) := \C\, \backslash\, \tilde{\rho} (L)$.
 \end{defi}
 
 Consider now the operator $L$ given in \eqref{equ-linear}, $ L = M + N$ with
 $$ M u = - e^{i \theta} u'' + i\omega u - k u, \quad D(M) = H^2(\R) $$
 and
 $$ N u = - e^{i \gamma} | \phi |^\sigma u  - 
 \sigma e^{i \gamma} | \phi |^{\sigma -2} \phi \,\Re \,(\overline{\phi}\, u) $$
 Notice that $N$ is a linear bounded operator on $L^2 (\R)$ (since $\phi \in H^1(\R) \subset L^\infty (\R)$).
 The next lemma determines, in a way, the location of the essential spectrum, $\sigma_e (L)$.
 \begin{lem}\label{lemkrein}
 $1.$ The operator $N (\lambda_0 - M)^{-1}$ is compact for some $\lambda_0 \in \R$.
 \medskip
 
\noindent  $2.$ Let $ k < 0$ and consider the half-plane $ \Sigma = \{ \Re \,\lambda < - k \} \subset \rho (M)$.
 Then, either $\Sigma$ consists of normal points of $M + N$ or consists entirely of eigenvalues
 of $M + N$, and so $\sigma_e (L) \subset \{ \Re \lambda \geq - k \} $.
 \end{lem}
 
 \noindent 
 \begin{proof} 1. Take a sequence $ \chi_j \in C_0^\infty (\R)$ such that $$\chi_j (x) = 1,\, x\in [-j, j],\,\,
\Supp \chi_j \subset [ -(j+1), j+1]$$ and consider the operators
$$ N_j u = - e^{i \gamma} | \phi_j |^\sigma u -  \sigma e^{i \gamma} | \phi_j |^{\sigma -2} \phi \,\Re \,(\overline{\phi_j}\, u) $$
where $\phi_j = \chi_j \phi$. For $\lambda_0 \in \Sigma,\, (\lambda_0 - M)^{-1} : L^2(\R) \rightarrow H^2(\R)$
and we see that, for any $j\in \N,\, N_j(\lambda_0 - M)^{-1}$ is a compact operator. On the other hand
$$ N_j (\lambda_0 - M)^{-1} \rightarrow N (\lambda_0 - M)^{-1}\,\, (j \rightarrow \infty) $$
in the operator norm of
 ${\mathcal L}\, (L^2 (\R))$. Since the space of compact operators is closed in the ${\mathcal L}\, (L^2 (\R))$
 space, it follows that  $N (\lambda_0 - M)^{-1}$ is compact.
 
 \noindent 2. We can now use \cite[Lemma 5.2]{krein} to conclude
 that the connected set $\Sigma \subset \rho (M)$ consists entirely of normal points of $L = M + N$
 or entirely of eigenvalues of $L$.
 \end{proof}
 
 \begin{proof}[Proof of the Theorem \ref{thm:estabilidadebs}] 
% 	First of all, when $k\to -\infty$, 
% 	$$ d \to \frac{ \sin \tilde{\theta} - 1} {\cos \tilde{\theta}}\in (-1,0)$$
% 	and the condition $|d|<\sqrt{1+\sigma/2}$ is satisfied. Moreover, it follows from \eqref{eq:condexistgamma} that $\gamma$ is well-defined. Therefore, the hypothesis of Theorem \ref{thm:existbs} are verified, ensuring the existence of a
% 	
% 	
 	Let $\phi=\psi\exp(id\ln\psi)$ be the bound-state built in Theorem \ref{thm:existbs}. Then $\psi$ satisfies
 	$$
 	\psi''=\epsilon \psi -\eta \psi^{\sigma+1},
 	$$
 	where $\epsilon$ and $\eta$ are given by \eqref{alfa}, \eqref{beta}. As a consequence (see \cite{Caz}, pag. 260),
 	$$
 	\|\phi\|_{L^\infty(\mathbb{R})}=  \|\psi\|_{L^\infty(\mathbb{R})}=|\psi(0)|=\left[ \epsilon \left( \frac{\sigma + 2} {2} \right) \right]^{\sigma/\epsilon}\eta^\sigma.
 	$$

  We have already seen (Lemma \ref{lemkrein}) that the essential spectum $\sigma_e (L)$
  lies in $ \{\Re \,\lambda > - k \}$. We need only to determine the location of the eigenvalues
  of $L = M + N$. Since $\| N u \|_{L^2(\R)} \leq (1 + \sigma) \| \phi \|_{L^\infty (\R)}^\sigma 
  \| u \|_{L^2(\R)}$, with $\phi$
  an $H^1$ solution of \eqref{equ-est2}, any eigenvalue $\lambda$ of $L$ verifies
  $$ \Re\, \lambda \geq - k - (1 + \sigma) \| \phi \|_{L^\infty (\R)}^\sigma $$
  and, therefore, one has the asymptotic stability of the bound-state $ e^{i t}\phi$ if
  \begin{equation}\label{eq:condestabilidade}
  (1 + \sigma) \| \phi \|_{L^\infty (\R)}^\sigma < - k.
  \end{equation}%  We recall that $\phi=\psi\exp(id\ln\psi)$, where $d$ is defined by \eqref{eq:valor-d},
%  $$
%  \psi''=\epsilon \psi -\eta \psi^{\sigma+1}
%  $$
%  and $\epsilon$ and $\eta$ given by \eqref{alfa}, \eqref{beta}. As a consequence (see \cite{Caz}, pag. 260),
%  $$
%  \|\phi\|_{L^\infty(\mathbb{R})}=  \|\psi\|_{L^\infty(\mathbb{R})}=|\psi(0)|=\left[ \epsilon \left( \frac{\sigma + 2} {2} \right) \right]^{\sigma/\epsilon}
%  $$
  Noticing that, when $k\to-\infty$,
  $$ \epsilon \sim -\frac{ k (1 + \cos\theta )} {2}  > 0\quad \mbox{ and }\quad  \eta\mbox{ bounded,} $$
  condition \eqref{eq:condestabilidade} is verified for sufficiently large $-k$.
  \end{proof}
  
  \section{Acknowledgements}
  The authors were partially supported by Funda\c{c}\~ao para a Ci\^encia e Tecnologia, through the grant UID/MAT/04561/2013.

 \bigskip
 \bigskip
 
 \normalsize
 
 \begin{center}
 	{\scshape Sim\~ao Correia}\\
 	{\footnotesize
 		CMAF-CIO, Universidade de Lisboa\\
 		Edif\'icio C6, Campo Grande\\
 		1749-016 Lisboa, Portugal\\
 		\email{sfcorreia@fc.ul.pt}
 	}
  \bigskip
 
 	 	{\scshape M\'ario Figueira}\\
 	{\footnotesize
 		CMAF-CIO, Universidade de Lisboa\\
Edif\'icio C6, Campo Grande\\
1749-016 Lisboa, Portugal\\
\email{msfigueira@fc.ul.pt}
 	}

\end{center}


\begin{thebibliography}{99}
  
  
  \bibitem{Caz} Cazenave, T., \emph{Semilinear Schr\"odinger Equations}, Courant Lecture Notes
  in Mathematics, {\bf 10}. New York University, American Mathematical Society, Providence, R.I., 2003.
 % \smallskip
  
  \bibitem{Caz-Weiss1}  Cazenave, T., Dickstein, F. and Weissler, F., \emph {Finite time blowup for
  a complex Ginzburg-Landau equation}, SIAM J. Math. Anal., 45 (2013), pp. 244-266.
 % \smallskip
  
   \bibitem{Caz-Weiss2}  Cazenave, T., Dickstein, F. and Weissler, F., \emph{Standing waves of the 
  complex Ginzburg-Landau equation}, Nonlinear Anal., 103 (2014), pp. 26-32.
  %\smallskip
  
  \bibitem{Haraux} Cazenave, T. and Haraux, A., \emph{An introduction to semilinear evolution equations},
  Oxford Lectures Series in Mathematics and its Applications, Oxford University Press, New York, 1998.
  %\smallskip
  
  \bibitem{Puel} Cipolatti, R., Dickstein, F and Puel, J-P., \emph{Existence of standing waves
  for the complex Ginzburg-Landau equation}, J. Math. Anal. Appl., 422 (2015), pp. 579-593.
  %\smallskip
  
  \bibitem{GV1} Ginibre, J. and Velo, G., \emph{The Cauchy problem in local spaces for 
  the complex Ginzburg-Landau equation {\rm I :} Compactness methods}, Phys. D, 95 (1996), pp. 191-228.
 % \smallskip
  
  \bibitem{GV2} Ginibre, J. and Velo, G., \emph{The Cauchy problem in local spaces for 
  the complex Ginzburg-Landau equation {\rm II :}  Contraction methods}, Comm. Math. Phys.,
  187 (1997), pp. 45-79.
  %\smallskip  
  
  \bibitem{krein} Gohberg, I.G. and Krein, M.G., \emph{Introduction to the Theory of Linear Nonselfadjoint
  Operators}, Transl. Math. Monog., Vol. {\bf 18} (1969) Am. Math. Soc.,
  Providence, R.I.  
 % \smallskip
  
  \bibitem{Hale} Hale, J.K., \emph{Dynamical systems and stability}, J. Math. Anal. Appl., 26 (1969), pp. 39-59.
 % \smallskip
  
  \bibitem{Henry} Henry, D., \emph{Geometric theory of semilinear parabolic equations}, Lecture
  Notes in Math., 840, Springer, New York, 1981.
  %\smallskip
  
  \bibitem{Lever} Levermore, C.D. and Oliver, M., \emph{The complex Ginzburg-Landau equation as a 
  model problem}, Lectures in Appl. Math. 31, AMS, Providence, R.I., 1996, pp. 141-190.
  %\smallskip
  
  \bibitem{Masmou} Masmoudi N. and Zaag H., \emph{Blow-up profile for the complex Ginzburg-Landau equation},
  J. Funct. Anal., 255 (2008), pp. 1613-1666.
 % \smallskip
  
  \bibitem{Okaza} Okazawa, N. and Yokota, T., \emph{Subdifferential operator approach to strong
  wellposedness of the complex Ginzburg-Landau equation}, Discrete Contin. Dyn. Syst., 28 (2010), pp. 311-341.
 % \smallskip
  
  \bibitem{Pazy} Pazy, A., \emph{Semi-groups of linear operators and applications to partial
  differential equations}, Springer, Berlin, 1983.

  \bibitem{zeidler} Zeidler, E., \emph{Nonlinear functional analysis and its applications. {I}}, Springer-Verlag, New York, 1986.
  
  




 
 
  
 
 
 \end{thebibliography}
\end{document}